\documentclass[12pt]{amsart}
\usepackage{latexsym}
\usepackage{enumerate}
\usepackage{amsmath, amssymb, amsfonts, amsxtra, amsthm}
\usepackage[active]{srcltx}

\author{Richard Oberlin}  
\title{Unions of lines in $F^n$}

\newcommand{\hdim}{\mathop{\mathrm{hdim}}}
\newcommand {\rea}{\mathbb{R}}

\newcommand {\spa}{\mathop{\mathrm{span}}}

\newcommand {\BBZ}{\mathbb{Z}}

\def\U{{U}}
\def\P{{\bf P}}

\def\L{{\bf L}}
\def\<{\left\langle}
\def\>{\right\rangle}

\newtheorem{theorem}{Theorem}[section]

\newtheorem{conjecture}[theorem]{Conjecture}
\newtheorem{proposition}[theorem]{Proposition}

\newtheorem{lemma}[theorem]{Lemma}

\begin{document}
\begin{abstract}
We show that if a collection of lines in a vector space over a finite field has ``dimension'' at least $2(d-1) + \beta,$ then its union has ``dimension" at least $d + \beta.$ This is the sharp estimate of its type when no structural assumptions are placed on the collection of lines. We also consider some refinements and extensions of the main result, including estimates for unions of $k$-planes. 
\end{abstract}
\maketitle
\section{Introduction}
The main problem we will consider here is to give a lower bound for the dimension of the union of a collection of lines in terms of the dimension of the collection of lines, without imposing a structural hypothesis on the collection (in contrast to the Kakeya problem where one assumes that the lines are direction-separated, or perhaps satisfy the weaker ``Wolff axiom''). 

Specifically, we are motivated by the following conjecture of D. Oberlin ($\hdim$ denotes Hausdorff dimension).
\begin{conjecture}
Suppose $d \geq 1$ is an integer, that $0 \leq \beta \leq 1,$ and that $\L$ is a collection of lines in $\rea^n$ with $\hdim(\L) \geq 2(d-1) + \beta.$ Then
\begin{equation} \label{dimrange}
\hdim(\bigcup_{L \in \L}L) \geq d + \beta.
\end{equation}
\end{conjecture}
The bound \eqref{dimrange}, if true, would be sharp, as one can see by taking $\L$ to be the set of lines contained in the $d$-planes belonging to  a $\beta$-dimensional family of $d$-planes. (Furthermore, there is nothing to be gained by taking $1 < \beta \leq 2$ since the dimension of the set of lines contained in a $d+1$-plane is $2(d-1) + 2$.)

Standard Fourier-analytic methods show that \eqref{dimrange} holds for $d=1$, but the conjecture is open for $d > 1.$ As a model problem, one may consider an analogous question where $\rea^n$ is replaced by a vector space over a finite-field. Our main result is that the corresponding conjecture holds for all $d$ ($|\cdot|$ denotes cardinality).

\begin{theorem} \label{firstlinethm}
Suppose $d \geq 1$ is an integer, $F$ is a finite field, $0 \leq \beta \leq 1,$ and that $\L$ is a collection of lines in $F^n$ with $|\L| \geq |F|^{2(d-1) + \beta}.$ Then
\begin{equation} \label{dimrangeff} 
|\bigcup_{L \in \L}L| \gtrsim |F|^{d + \beta}
\end{equation}
where the implicit constant may depend on $d$, but is independent\footnote{The constant is also independent of $\beta$ and $n$, but this is only of secondary interest.} of $F$.
\end{theorem}

Reusing the examples above, one sees that \eqref{dimrangeff} is sharp, up to the loss in the implicit constant, and that there is nothing to be gained by taking $1 < \beta \leq 2.$

The main tool we use in the proof of \eqref{dimrangeff} is an iterated version of Wolff's hairbrush argument \cite{wolff95ibk}. For comparison, we state the finite-field version of his result\footnote{Wolff's main interest in this method was likely its use towards a partial resolution of the Kakeya conjecture (up to a negligible constant, any direction separated collection of lines satisfies the Wolff axiom). To that end, it has been superceded by Dvir's theorem \cite{dvir09osk} (see also \cite{ellenberg09ksm}), whose proof makes stronger use of the direction-separation hypotheses and does not seem to be applicable to the present question.} (see \cite{wolff99rwc},\cite{mockenhaupt04rkp}), starting with the following definition. A set of lines $\L$ in $F^n$ satisfies the \emph{Wolff axiom} if for every two-plane $R \subset F^n$ 
\[
|\{L \in \L : L \subset R\}| < |F|.
\]

\begin{theorem}[Wolff]
Suppose that $\alpha \geq 1$, $F$ is a finite field, and $\L$ is a collection of lines in $F^n$ with $|\L| \geq |F|^{\alpha}.$ If $\L$ satisfies the Wolff axiom then
\begin{equation} \label{wolffexponent}
|\bigcup_{L \in \L}L| \gtrsim |F|^{\frac{\alpha + 3}{2}}
\end{equation}
where the implicit constant is independent of $F$.
\end{theorem}

An immediate consequence of Theorem \ref{firstlinethm} is that, for odd integers $\alpha$, Wolff's theorem holds even for collections of lines that do not satisfy the Wolff axiom. 

The proof of Theorem \ref{firstlinethm} also shows that the Wolff axiom can be relaxed for general values of $\alpha.$ We say that a set of lines $\L$ in $F^n$ satisfies the \emph{$d$-plane Wolff axiom} if for every $d$-plane $R \subset F^n$ 
\begin{equation} \label{dpwa}
|\{L \in \L : L \subset R\}| < |F|^{2d-3}.
\end{equation}
If $R$ is a $d$-plane then there are approximately $|F|^{3(d-2)}$ 2-planes $S$ contained in $R$ and for each line $L \subset R$ there are approximately $|F|^{d-2}$ 2-planes $S$ with $L \subset S \subset R.$ Thus, the $d$-plane Wolff axiom asserts that for every $d$-plane $R$ the standard Wolff axiom holds ``on average'' for two-planes $S \subset R.$ In particular, the d-plane Wolff axiom is weaker than the standard Wolff axiom when $d > 2$ (assuming one is willing to adjust the axioms by a constant factor, which would make no impact on the validity of the stated theorems). 

\begin{theorem} \label{wolffrelaxed}
Suppose that $d > 2$ is an integer, $2(d-1) - 1 < \alpha < 2(d-1) + 1$, $F$ is a finite field, and $\L$ is a collection of lines in $F^n$ with $|\L| \geq |F|^{\alpha}.$ If $\L$ satisfies the $d$-plane Wolff axiom \eqref{dpwa} then
\begin{equation} \label{wolffexponent2}
|\bigcup_{L \in \L}L| \gtrsim |F|^{\frac{\alpha + 3}{2}}
\end{equation}
where the implicit constant is independent of $F$.
\end{theorem}

Bounds of the form \eqref{wolffexponent} do not seem to be sharp; at least, they can be slightly strengthened in the case when $F = \BBZ_p$, $\alpha = 2$, and $n=3$, see \cite{bourgain04spe}. 

Since \eqref{wolffexponent2} improves on \eqref{dimrangeff} when $\beta < 1$, one can use the $d$-plane Wolff axiom to extract structural information about quasi-extremizers (cf. \cite{christ06qer}) of \eqref{dimrangeff}.

\begin{theorem} \label{qethm}
Suppose $d \geq 1$ is an integer, and that $0 \leq \beta < 1$. Then, for every $C$ there exist $M$ and $c > 0$ such that if $F$ is a finite field with $|F| \geq M$ and $\L$ is a collection of lines in $F^n$ with $|L| \geq |F|^{2(d-1) + \beta}$ satisfying
\begin{equation} \label{qebound}
|\P| \leq C  |F|^{d + \beta}
\end{equation}
where $\P := \bigcup_{L \in \L}L,$
then there are $d$-planes $R_1, \ldots, R_N$ with $N \geq c |F|^\beta$ such that 
\[
|\P \cap \bigcup_{j} R_j| \geq c |\P|
\]
and for each $j$
\begin{equation} \label{denseindplane}
|\P \cap R_j| \geq c |R_j|.
\end{equation}
\end{theorem}

One can also prove a version of the statement above for $-1 < \beta < 0,$ but we omit the details. 

By adding two additional layers of recursion, the method of Theorem \ref{firstlinethm} can be adapted to treat unions of $k$-planes. 

\begin{theorem} \label{upthm1}
Suppose $d \geq k > 0$ are integers, that $0 \leq \beta \leq 1$ and that $\L$ is a collection of $k$-planes in $F^n$ with 
\begin{equation} \label{upthm1hyp}
|\L| \geq  |F|^{(k+1)(d-k)+\beta}.
\end{equation}
Then 
\begin{equation} \label{unionplanesbound}
|\bigcup_{L \in \L}L| \gtrsim |F|^{d + \beta}.
\end{equation}
\end{theorem}

Our proof requires simultaneous treatment of the following more general result.

\begin{theorem} \label{upthm2}
Suppose $d \geq k > k' \geq 0$ are integers, that $0 \leq \beta \leq k' + 1$ and that $\L$ is a collection of $k$-planes in $F^n$ with 
\[
|\L| \geq  |F|^{(k+1)(d-k)+\beta}.
\]
Letting $\P_L = \{k'\text{-planes\ }P : P \subset L\}$ we have
\begin{equation} 
|\bigcup_{L \in \L}\P_L| \gtrsim |F|^{(k' + 1)(d - k') + \beta}.
\end{equation}
\end{theorem}

Theorems \ref{upthm1} and \ref{upthm2} are sharp in the same sense as Theorem \ref{firstlinethm}. For a $k$-plane analog of Wolff's theorem see \cite{bueti06ibk}. It may be possible to modify the proof of \eqref{unionplanesbound} to obtain a $k$-plane analog of Theorem \ref{wolffrelaxed}, but we do not pursue the details here.

The outline of this article is as follows: Section \ref{prsection} contains some technical machinery, Section \ref{linesection} contains the proofs of Theorems \ref{firstlinethm}, \ref{wolffrelaxed}, and \ref{qethm},  and Section \ref{planesection} contains the proofs of Theorems \ref{upthm1} and \ref{upthm2}.

\section{Preliminaries} \label{prsection}

We start by roughly describing the approach of \cite{wolff95ibk}. If a union of lines is small, then there must exist a ``hairbrush'' of many lines intersecting one common line. The ambient space can then be foliated into two-dimensional planes containing the common line, and a classical bound can be applied to estimate the union of lines contained in each two-plane.  

In the present situation, we instead consider a hairbrush of many lines or $k$-planes intersecting a common $m$-dimensional plane. The following lemma is used to determine the appropriate choice of $m$.

\begin{lemma} \label{pslemma}
Let $\L$ be a collection of $k$-planes in $F^n$ and suppose $d$ is a nonnegative integer with $k \leq d \leq n$. There is an $m$ with $k \leq m \leq d$, a collection of $m$-planes $R_1, \ldots, R_N$, and collections of $k$-planes $\L_{R_1}, \ldots, \L_{R_N}$ such that
\begin{enumerate}[(a)]
\item the $\L_{R_j}$ are pairwise disjoint subsets of $\L$ with $L \subset R_j$ for $L \in \L_{R_j}$;
\item if $m > k$ then $|\L_{R_j}| \geq  |F|^{(k+1)(m - 1 - k) + k}$;
\item if $m=k$ then $|\L_{R_j}| = 1$;
\item \label{epcondition} letting $\L^m = \bigcup_{j}\L_{R_j}$, we have $|\L^m| \geq 2^{-(d-m+1)}|\L|$;
\item \label{genwolffcondition} if $m < m' \leq d$ then for every $m'$-plane $S$, $|\{L \in \L^m : L \subset S\}| < |F|^{(k+1)(m' - 1 - k) + k} .$
\end{enumerate}
\end{lemma}

\begin{proof}
Set $\L^{*,d+1} := \L.$ For $k \leq m \leq d$, suppose that $\L^{*,m+1} \subset \L$ has been chosen so that $|\L^{*,m+1}| \geq 2^{-(d-m)} |\L|$. Starting at $j = 1$, suppose $m$-planes $R^m_{j'}$ and collections of $k$-planes $\L_{R^m_{j'}}$ have been selected for all integers $0 < j' < j$.

If there is an $m$-plane $R$ so that 
\[
|\{L \in \L^{*,m+1} \setminus \bigcup_{j'  < j} \L_{R^m_{j'}} : L \subset R\}| \geq |F|^{(k+1)(m-1-k) + k}
\]
then let $R^m_j$ be such an $m$-plane, let
\[
\L_{R^m_j} = \{L \in \L^{*,m+1} \setminus \bigcup_{j'  < j} \L_{R^m_{j'}} : L \subset R^m_j\}
\]
and continue the process with $j+1.$ If there is no such plane then terminate the process and set 
\[
\L^{*,m} = \L^{*,m+1} \setminus \bigcup_{j'  < j} \L_{R^m_{j'}}.
\]
If $|\L^{*,m}| < |\L^{*,m+1}|/2$ then property (\ref{epcondition}) is satisfied and we terminate the process. Otherwise, continue with $m-1.$

If the process reaches the stage $m=k,$ then let $R^{k}_1, \ldots, R^k_N$ be some enumeration of $\L^{*,k+1}$ and $\L_{R^{k}_j} = \{R^k_j\}$ and we are finished.
\end{proof}

Next, we describe in detail the foliation of the ambient space.

\begin{lemma} \label{folprop}
Suppose that $S$ is an $m$-plane in $F^n$, that $0 \leq q \leq k-1,$ and that $m + k-q \leq n.$ Then we can find $(m+k-q)$-planes $T_1, \ldots, T_N$ with $S \subset T_i$ for all $i$ such that for all $(k-1)$-planes $P$ and $k$-planes $L$ satisfying
\begin{enumerate}[(a)]
\item $P \subset L$
\item $L \cap S$ is a $q$-plane
\item $P \cap S$ is a $(q-1)$-plane if $q > 0$ and $P \cap S = \emptyset$ if $q = 0$
\end{enumerate}
we have $L \subset T_i$ for some $i$ and $P \not\subset T_{i'}$ for $i' \neq i.$
\end{lemma}

\begin{proof}
To find the $T_i$, write $S=x + \spa(e_1, \ldots, e_m)$ and 
\[
F^n = \spa(e_1, \ldots, e_m, f_1, \ldots f_{n-m}).
\]
Then write $T_i = S + V_i$ where, as $i$ varies, $V_i$ ranges over all $(k-q)$-dimensional subspaces of $\spa(f_1, \ldots, f_{n-m}).$

Fix some $P,L$ satisfying the hypotheses. One can check that there is an $i$ such that $L \subset T_i.$ For any $i' \neq i$ we have that $T_i \cap T_{i'}$ contains $S$ and is, at most, an $(m +  k-q - 1)$-plane. 

First consider the case $q > 0.$ Choose $y \in P \cap S$ and write 
\begin{align*}
P \cap S &= y + \spa(g_1, \ldots, g_{q-1}),\\
P &= y + \spa(g_1, \ldots, g_{q-1}, h_1, \ldots, h_{k-q}),\\
\intertext{and}
S &= y + \spa(g_1, \ldots, g_{q-1}, h'_{1}, \ldots, h'_{m + 1 -q}).
\end{align*}
Clearly, 
\[
W := \{g_1, \ldots, g_{q-1}, h_1, \ldots, h_{k-q}, h'_{1}, \ldots, h'_{m + 1 -q } \}
\]
is linearly independent and, since $S \subset T_i \cap T_{i'}$, $P \subset T_i \cap T_{i'}$ would imply $y + \spa(W) \subset T_i \cap T_{i'}$ contradicting the dimension estimate on the latter set.  

For the case $q=0$ write 
\begin{align*}
P &= z + \spa(h_1, \ldots, h_{k-1})\\
\intertext{and}
S &= y + \spa(h'_1, \ldots, h'_m).
\end{align*}
Again using the dimension estimate on $T_i \cap T_{i'}$, we see that if $P \subset T_{i} \cap T_{i'}$ then, since $P \cap S = \emptyset,$ we have $v \in \spa(h_1', \ldots, h'_m)$ for some $0 \neq v \in \spa(h_1, \ldots, h_{k-1}).$ But, since $L \cap S \neq \emptyset$, this implies $\dim(L \cap S) > 0$, contradicting the assumption that $\dim(L \cap S) = 0.$
\end{proof}

To estimate the union of lines or $k$-planes contained in each leaf of the foliation, we will appeal to recursion. However,  at the root we still use the classical method:
 
\begin{lemma} \label{classmethod}
Suppose that $\L$ is a collection of $m$-planes, $\P$ is a collection of $(k-1)$-planes such that for every $L \in \L$
\[
|\{P \in \P : P \subset L\}| \geq M,
\]
and 
\begin{equation} \label{mustpruneeq}
|\L| |F|^{k(m-k)} \leq M.
\end{equation}
Then
\[
|\P| \gtrsim |\L| M.
\]
\end{lemma}

\begin{proof}
For each $L \in \L$, let $\P_L$ be a subset of $\{P \in \P : P \subset L\}$ with $M \leq \P_L \leq 2M.$ Set
\[
\U = \{(P,L,L')  : (L,L') \in \L^2, P \in \P_L \cap \P_{L'} \}.
\]
An application of Cauchy-Schwarz gives 
\[
|\U| \geq \frac{M^2|\L|^2}{|\P|}.
\]
Any two distinct $m$-planes intersect in, at most, an $(m-1)$-plane. Since, by  \eqref{plcount} below, an $(m-1)$-plane contains $\lesssim |F|^{k(m-k)}$ $(k-1)$-planes, we have 
\begin{align*}
|\U| &\leq C |\L|^2 |F|^{k(m-k)} + |\{(P,L) : L \in \L, P \in \P_L\}| \\
&\leq C|\L|^2|F|^{k(m-k)} + |\L| 2 M \\
&\lesssim |\L|M.
\end{align*}
\end{proof}

We finish the section with three standard estimates for collections of planes.  

\begin{lemma} \label{grascount}
For integers $0 \leq k \leq d$ we have
\begin{equation} \label{sscount}
|G(d,k)| \approx |F|^{k(d-k)}
\end{equation}
and
\begin{equation} \label{plcount}
|G'(d,k)| \approx |F|^{(k+1)(d-k)}
\end{equation}
where $G(d,k)$ is the set of $k$-dimensional subspaces of $F^d$ and $G'(d,k)$ is the set of $k$-planes in $F^d.$ 
\end{lemma}

\begin{proof}
A generic choice of $k$ vectors in $F^d$ is linearly independent, and there are approximately $|F|^{kd}$ such choices. By the same logic, each $k$-plane has approximately $F^{k^2}$ choices of basis, and hence we have \eqref{sscount}. 

Given a $k$-dimensional subspace $P$ with basis $e_1, \ldots, e_k$, choose \\ $f_1, \ldots, f_{d-k}$ so that $F^d = \spa(e_1, \ldots, e_k, f_1, \ldots, f_{d-k})$. Then there is a one-to-one correspondence between linear combinations of $f_1, \ldots, f_{d-k}$ and distinct translates of $P$, giving \eqref{plcount}.
\end{proof}

\begin{lemma} \label{sccprop}
Suppose that $S$ is an $l$-plane in $F^m.$ Then for $l < l' \leq m$
\begin{equation} \label{scontainedcount}
|\{P \subset F^m : P \text{\ is a\ } l'\text{-plane\ and\ } S \subset P\}| \lesssim |F|^{(l' - l)(m - l')}.
\end{equation}
\end{lemma}

\begin{proof}
Write $S = x + \spa(e_1, \ldots, e_l)$ and 
\[
F^m = \spa(e_1, \ldots, e_l, f_1, \ldots, f_{m-l}).
\] 
Then there is a one-to-one correspondence between $l'$-planes $P \supset S$ and $(l'-l)$-dimensional subspaces of $\spa(f_1, \ldots, f_{m-l})$, so \eqref{scontainedcount} follows from \eqref{sscount}.
\end{proof}

\begin{lemma} \label{eiprop}
Suppose $S$ is an $l$-plane in $F^k.$ Then
\[
|\{P \subset F^k : P \text{\ is a\ } (k-1)\text{-plane\ and\ } P \cap S = \emptyset\}| \lesssim |F|^{k-l}.
\]
\end{lemma}

\begin{proof}
Write $S = x + \spa(e_1, \ldots, e_l)$, and fix $P = y + \spa(f_1, \ldots, f_{k-1}).$ If $P \cap S = \emptyset$, we must have $e_j \in \spa(f_1, \ldots, f_{k-1})$ for $j = 1, \ldots, l.$ Thus, $P$ is a translate of a $(k-1)$-plane containing $S$. Since, by Lemma \ref{sccprop}, there are $\lesssim |F|^{k-l-1}$ $(k-1)$-planes containing $S$, we have at most $|F|^{k-l}$ possible planes $P$. 
\end{proof}

\section{Unions of lines} \label{linesection}
Theorems \ref{firstlinethm} and \ref{wolffrelaxed} follow immediately from:
\begin{proposition} \label{linetheorem}
Suppose $d \geq 1$ is an integer, that $0 < \gamma,\lambda \leq 1$, that $\max(1 - d,-1) \leq \beta \leq 1,$ that $\L$ is a collection of lines in $F^n$ with 
\[
|\L| \geq \gamma |F|^{2(d-1)+\beta},
\]
and that $\P$ is a collection of points in $F^n$ satisfying 
\[
|\{P \in \P : P \in L\}| \geq \lambda |F|
\]
for every $L \in \L.$ Then
\begin{equation} \label{nondpwaresult}
|\P| \gtrsim  |F|^{d +\max(0,\beta)}.
\end{equation}
where the implicit constant may depend on $d,\gamma,\lambda.$ Furthermore, if $d \geq 2$ and $\L$ satisfies the $d$-plane Wolff axiom \eqref{dpwa} then we have
\begin{equation} \label{dpwaresult}
|\P| \gtrsim |F|^{d + \frac{\beta + 1}{2}}.
\end{equation}
\end{proposition}

\begin{proof}[Proof of Theorem \ref{qethm} assuming Proposition \ref{linetheorem}]
Suppose that $\L$ satisfies \eqref{qebound} and that $|F| \geq M$ where $M$ is large and to be determined later.

 For $d$-planes $R$ let $\L_R = \{L \in \L : L \subset R\}$. Choose $d$-planes $R_1, \ldots, R_N$ so that for each $j$, $|\L_{R_j}| \geq |F|^{2d-3}$ and such that for $R \not\in \{R_j\}_{j=1}^N$ we have $|\L_{R}| < |F|^{2d-3}. $

Setting $\L' = \bigcup_{j}\L_{R_j}$, $\L'' = \L \setminus \L'$ and $\P'' = \bigcup_{L \in \L''}L$, we must have $|\L''| < \frac{1}{2}|\L'|$ or 
else we would have
\[
|\P| \geq |\P''| \geq c' |F|^{d + \frac{\beta + 1}{2}} > C |F|^{d+\beta}
\]
where $c'$ is the implicit constant from \eqref{dpwaresult} and $M$ is chosen large enough to overwhelm $\frac{C}{c'}$ (In the second inequality above we have used the fact that $\L''$ satisfies \eqref{dpwa} to obtain \eqref{dpwaresult} from Proposition \ref{linetheorem}.) 

Thus, $|\L'| \geq \frac{1}{2}|\L|$ and, by \eqref{plcount} with $k=1$, $N \geq c |F|^{\beta}.$ Letting $\P' = \bigcup_{L \in \L'}L$ we have
\[
|\P'| \geq c'' |F|^{d + \beta} \geq c C|F|^{d + \beta} \geq c |\P|
\]
where $c''$ is the implicit constant from \eqref{nondpwaresult}, $c$ is chosen small enough to underwhelm $\frac{c''}{C},$ and the last inequality follows from \eqref{qebound}.

A final application of $\eqref{nondpwaresult}$ with $d' = d-1$ and $\beta' = 1$ then gives \eqref{denseindplane} since $|\L_{R_j}| \geq |F|^{2(d-1-1)+1}.$
\end{proof}

\begin{proof}[Proof of Proposition \ref{linetheorem}]
We use induction on $d$. When $d=1$ the result follows from an application of Lemma \ref{classmethod} (with $m=1$, $k=1$, and $\L$ replaced by a subset of itself with cardinality $\approx \min(\gamma,\lambda) |F|^\beta$ ), so we will prove it for $d > 1$, working under the assumption that it has already been proven for $1 \leq d' < d.$

After possibly deleting lines, we may assume
\begin{equation} \label{refinedkk1}
|\L| < 2\gamma |F|^{2(d-1)+\beta}.
\end{equation}
Applying Lemma \ref{pslemma} to $\L$ we obtain $m$-planes $R_1, \ldots, R_N.$ Note that if $L$ satisfies the $d$-plane Wolff axiom \eqref{dpwa} then we must have $m < d.$ 

\ 

\noindent{\bf Case 1, $(m=d)$:}\\ 
Let $\P_{R_j} = \{P \in \P : P \in R_j\}.$ Applying the case $d' = d-1$ of the proposition to the $\L_{R_j}$, we deduce that $|\P_{R_j}| \gtrsim |F|^{d}$ for each $j$. Since, by \eqref{plcount}, $|\L_{R_j}| \lesssim |F|^{2(d-1)}$, we have $N \gtrsim |F|^{\max(0,\beta)}.$ Then using Lemma \ref{classmethod} (possibly applying it to a subset of $\{R_j\}_{j=1}^N$ in order to satisfy \eqref{mustpruneeq}) to estimate $|\bigcup_{j}\P_{R_j}|$ shows that 
\[
|\P| \gtrsim  |F|^{d + \max(0,\beta)}
\]
as desired.
 
\ 

\noindent{\bf Case 2, $(m < d)$:}\\
We construct sets of lines and points using a standard ``popularity'' argument. Fix some large $C$ to be determined later and let 
\[
\P^{\sharp} = \{P \in \P : |\{L \in \L : P \in L\}| \geq C |F|^{d-1 - \frac{1-\beta}{2}}\}
\]
and
\[
\L^{\sharp} = \{L \in \L : |\{P \in \P^{\sharp} : P \in L\}| \geq \frac{1}{4} \lambda |F|\}.
\]
Letting $\L^m$ be as in Lemma \ref{pslemma}, we then have either 
\begin{align} \label{firstpossibilityk1}
|\P| &\geq \frac{1}{2}\lambda |F|  |\L^m|/(C|F|^{(d-1) - \frac{1-\beta}{2}})
\end{align}
(in which case we are finished since the right hand side above is $\gtrsim |F|^{d + \frac{\beta+1}{2}}$) or
\begin{equation} \label{secondpossibilityk1}
|\L^{\sharp}| \geq \frac{1}{8} |\L^m|
\end{equation}

Indeed, suppose that \eqref{firstpossibilityk1} does not hold. Set
\[
I = \{(P,L) : P \in \P_L, L \in \L^{m}\}
\]
where, for each $L$, $\P_L$ is a subset of $\P \cap L$ with $\lambda |F| \leq |\P_L|  < 2 \lambda |F|.$ 
Then letting 
\[
I' = \{(P,L) : P \in \P_L \setminus \P^{\sharp} , L \in \L^{m}\}
\]
we have
\[
|I'| < C |F|^{d-1 - \frac{1-\beta}{2}} |\P| < \frac{1}{2} |I|
\]
and so
\[
|\{(P,L) : P \in \P_L \cap \P^{\sharp}, L \in \L^{m}\}| \geq \frac{1}{2} \lambda |F| |\L^{m}|
\]
giving 
\[
|\{(P,L) : P \in \P_L \cap \P^{\sharp}, L \in \L^{\sharp}\}| \geq \frac{1}{4} \lambda |F||\L^{m}|
\]
thus leading (by the upper bound on $|\P_L|$) to \eqref{secondpossibilityk1} as claimed.

Let $\L^{\sharp}_{R_j} = \L_{R_j} \cap \L^{\sharp}$, $\P^{\sharp}_{R_j} = R_j \cap \P^{\sharp}$, 
\[
\L'_{R_j} = \{L \in \L^m : |L \cap R_j| = 1\},
\]
and
\[
\P'_{R_j} = \P \setminus R_j.
\]
Fix $j$ so that $|\L^{\sharp}_{R_j}| \gtrsim |\L_{R_j}|,$ and recall $|\L_{R_j}| \geq |F|^{2m-3}$ by Lemma \ref{pslemma}.
Applying the previously known case $d' = m-1$ of the proposition (or using the trivial estimate if $m=1$) we have
\begin{equation} \label{lequalszeroPk1}
|\P^{\sharp}_{R_j}| \gtrsim |F|^{m}.
\end{equation}

For each point $P \in \P^{\sharp}_{R_j}$ there are $\geq C |F|^{d-1 -  \frac{1-\beta}{2}}$ lines from $\L^{m}$ intersecting $P$. Since, by Lemma \ref{sccprop}, $\lesssim |F|^{m-1}$ of these lines are contained in $R_j$, we have
\[
|\{L \in \L'_{R_j} : P \in L\}| \geq \frac{1}{2} |F|^{d - 1 - \frac{1-\beta}{2}}
\]
provided that $C$ is chosen sufficiently large. 
Thus, 
\begin{align*}
|\L'_{R_j}| &\gtrsim |\P^{\sharp}_{R_j}| |F|^{d - 1 - \frac{1 - \beta}{2}} \\
&\gtrsim |F|^{m + d-1 - \frac{1-\beta}{2}}.
\end{align*}
Note\footnote{We may assume throughout that $|F|$ is sufficiently large relative to certain parameters (for instance $\lambda$) since the implicit constants may be chosen so that the conclusion holds trivially for small $|F|$.} that for each $L \in \L'_{R_j}$, $|L \cap \P'_{R_j}| \geq \lambda |F| - 1 \gtrsim |F|.$

Applying Lemma \ref{folprop}, we write $F^n$ as the union of $(m + 1)$-planes $T_i$ containing $R_j$. Let
\begin{align*}
\L_i &= \{L \in \L'_{R_j} : L \subset T_i\} \\ 
\P_i &= \{P \in \P'_{R_j} : P \in T_i\}.
\end{align*}
Then 
\begin{align*}
|\P| &\geq \sum_i |\P_i| \\ 
&\gtrsim \sum_i |\L_i|/|F|^{m-2} \\ 
&\geq |\L'_{R_j}|/|F|^{m - 2}\\ 
&\gtrsim |F|^{d +\frac{\beta+1}{2}}
\end{align*}
where, for the second inequality, we used the fact (which follows from Lemma \ref{pslemma}) that $|\L_i| < |F|^{2(m -1) + 1} $ to see that $|\L_i| = |F|^{2(d' - 1) + \beta'}$ for some $d' \leq m$ and so we can estimate $|\P_i|$ using the previously known case $d'$ of \eqref{nondpwaresult}.
\end{proof}

\section{Unions of Planes} \label{planesection}
Theorem \ref{upthm2} is obtained by induction from the hyperplane case:
\begin{proposition} \label{maintheoremcd1}
Suppose $d,k > 0$ are integers, that $0 < \gamma,\lambda \leq 1$, that  $d \geq k$, that $0 \leq \beta \leq k$, that $\L$ is a collection of $k$-planes in $F^n$ with 
\[
|\L| \geq \gamma |F|^{(k+1)(d-k)+\beta},
\]
and that $\P$ is a collection of $(k-1)$-planes in $F^n$ satisfying 
\[
|\{P \in \P : P \subset L\}| \geq \lambda |F|^{k}
\]
 for every $L \in \L.$ Then
\[
|\P| \gtrsim  |F|^{k(d - k + 1)+\beta}
\]
where the implicit constant may depend on $d,\gamma,\lambda, k.$
\end{proposition}

\begin{proof}[Proof of Theorem \ref{upthm2} assuming Proposition \ref{maintheoremcd1}]
Theorem \ref{upthm2} follows directly from Proposition \ref{maintheoremcd1} when $k-k' = 1$. Fix $k_0 \geq 1$ and assume that the theorem holds for all $k-k' = k_0,$ and fix some $k,k'$ with $k-k' = k_0 + 1.$ 

For $L \in \L$ satisfying \eqref{upthm1hyp} let 
\[
\P'_L = \{(k'+1)\text{-planes\ }P' : P' \subset L\}
\] 
and for any $(k'+1)$-plane $P'$ let 
\[
\P''_{P'} = \{k'\text{-planes\ }P : P \subset P'\}.
\]
Applying the previously known case of the theorem, we have
\[
|\bigcup_{L \in \L}\P'_L| \gtrsim |F|^{(k' + 2)(d - (k' + 1)) + \beta}
\]
and thus a second application of Proposition \ref{maintheoremcd1} gives
\[
|\bigcup_{L \in \L}\P_L| = |\bigcup_{L} \bigcup_{P' \in \P'_L} \P''_{P'}| \gtrsim |F|^{(k' + 1)(d - k') + \beta}.
\]
\end{proof}

\begin{proof}[Proof of Proposition \ref{maintheoremcd1}]
We use induction on $d$. When $d=k$ the result follows from an application of Lemma \ref{classmethod} (with $m=k$). So we will prove it for $d > k$, working under the assumption that it has already been proven for $k \leq d' < d.$

After possibly deleting planes, we may assume
\begin{equation} \label{refinedk}
|\L| < 2\gamma |F|^{(k+1)(d-k)+\beta}.
\end{equation}
Applying Lemma \ref{pslemma} to $\L$ we obtain $m$-planes $R_1, \ldots, R_N.$

\ 

\noindent{\bf Case 1, $(m=d)$:}\\ 
Let $\P_{R_j} = \{P \in \P : P \subset R_j\}.$ Applying the case $d' = d-1$ of the theorem to the $\L_{R_j}$, we deduce that $|\P_{R_j}| \gtrsim |F|^{k(d-k+1)}$ for each $j$. Since, by \eqref{plcount}, $|\L_{R_j}| \lesssim |F|^{(k+1)(d-k)}$, we have $N \gtrsim |F|^{\beta}.$ Using Lemma \ref{classmethod} to estimate $|\bigcup_{j}\P_{R_j}|,$ we conclude
\[
|\P| \gtrsim  |F|^{k(d-k+1)+\beta}
\]
as desired.
 
\ 

\noindent{\bf Case 2, $(m < d)$:}\\
We construct sets of $k$-planes and $(k-1)$-planes using a standard ``iterated-popularity'' argument. Fix some large $C$ to be determined later. Let $\L^{\sharp,0} = \L^m$, $\P^{\sharp,0} = \P$ and for $1 \leq q \leq k$ let
\[
\P^{\sharp,q} = \{P \in \P^{\sharp,q-1} : |L \in \L^{\sharp, q-1} : P \subset L| \geq C|F|^{d-k}\}
\]
and
\[
\L^{\sharp,q} = \{L \in \L^{\sharp,q-1} : |P \in \P^{\sharp,q} : P \subset L| \geq 2^{-2q} \lambda |F|^k\}.
\]
Then either 
\begin{align} \label{firstpossibility}
|\P| &\geq 2^{-5k}\lambda |F|^k  |\L^m|/(C|F|^{d-k})
\end{align}
(in which case we are finished since the right hand side above is $\gtrsim |F|^{k(d-k+1) + \beta}$) or for each $q \leq k$
\begin{equation} \label{secondpossibility}
|\L^{\sharp,q}| \geq 2^{-3q} |\L^m|.
\end{equation}

Indeed, suppose that \eqref{firstpossibility} does not hold and that \eqref{secondpossibility} holds for $0 \leq q \leq q_0 < k$. Set 
\[
I = \{(P,L) : P \in \P_L, L \in \L^{\sharp,q_0}\}
\]
where, for each $L$, $\P_L$ is a subset of $\P^{\sharp,q_0}$ with $P \subset L$ for every $P \in \P_L$ and $2^{-2q_0} \lambda |F|^k \leq |\P_L|  < 2^{-(2q_0-1)} \lambda |F|^k.$ Note
\[
|I| \geq 2^{-2q_0} \lambda |F|^k |\L^{\sharp,q_0}| \geq 2^{-5q_0} \lambda |F|^k |\L^m|.
\]
Then letting 
\[
I' = \{(P,L) : P \in \P_L \setminus \P^{\sharp,q_0 + 1} , L \in \L^{\sharp,q_0}\}
\]
we have
\[
|I'| < C|F|^{d-k} |\P| \leq 2^{-5(k - q_0)} |I| \leq \frac{1}{2} |I|
\]
and so
\[
|\{(P,L) : P \in \P_L \cap \P^{\sharp,q_0 + 1}, L \in \L^{\sharp,q_0}\}| \geq \frac{1}{2} 2^{-2q_0} \lambda |F|^k |\L^{\sharp,q_0}|.
\]
This gives
\[
|\{(P,L) : P \in \P_L \cap \P^{\sharp,q_0 + 1}, L \in \L^{\sharp,q_0+1}\}| \geq \frac{1}{4} 2^{-2q_0} \lambda |F|^k |\L^{\sharp,q_0}|
\]
thus leading (by the upper bound on $|\P_L|$) to 
\[
|\L^{\sharp,q_0 + 1}| \geq \frac{1}{8} |\L^{\sharp,q_0}|
\]
as claimed.

For each $1 < q \leq k$ let 
\begin{align*}
\L^{\sharp,q}_{R_j} &= \{L \in \L^{\sharp,q} : L \cap R_j \text{\ is a }q\text{-plane}\} \\
\P^{\sharp,q}_{R_j} &= \{P \in \P^{\sharp,q} : P \cap R_j \text{\ is a }(q-1)\text{-plane}\}.
\end{align*}
Define $\L^{\sharp,0}_{R_j}$ as above and let $\P^{\sharp,0}_{R_j} = \{P \in \P^{\sharp,0} : P \cap R_j = \emptyset\}.$

Letting $\L_{R_j}$ be as in Lemma 2.1, fix $j$ so that $|\L^{\sharp,k}_{R_j}| \gtrsim |\L_{R_j}|.$
Applying the previously known case $d' = m-1$ of the theorem (or using the trivial estimate if $m=k$) we have
\begin{equation} \label{lequalszeroP}
|\P^{\sharp,k}_{R_j}| \gtrsim |F|^{k(m-k+1)}.
\end{equation}

Suppose for some $1 \leq q \leq k$ that  
\begin{equation} \label{equationA}
|\P^{\sharp,q}_{R_j}| \gtrsim |F|^{\omega}.
\end{equation}
Our immediate goal is to see that under certain conditions, \eqref{equationA} implies \eqref{equationB} and \eqref{equationstar} below. For each $P \in \P^{\sharp,q}_{R_j}$ there are $\geq C |F|^{d-k}$ $k$-planes from $\L^{\sharp,q-1}$ containing $P$. Since, by Lemma \ref{sccprop}, $\lesssim |F|^{m-q}$ of these $k$-planes intersect $R_j$ in $q$-planes (and there is no possibility that for a $k$-plane $L \supset P$ we have $\dim(L \cap R_j) > q$, since $\dim(L) = \dim(P) + 1$), we have
\[
|\{L \in \L^{\sharp,q-1}_{R_j} : P \subset L\}| \geq \frac{1}{2} C |F|^{d - k}
\]
provided that $C$ is chosen sufficiently large and 
\begin{equation} \label{firstiterationrequirement}
m + k-q \leq d . 
\end{equation}
Thus, using the fact (which follows from Lemma \ref{sccprop}) that for each $(q-1)$-plane $S \subset L$ there are at most $|F|^{k-q}$ $(k-1)$-planes $P$ with $S \subset P \subset L$, we have
\begin{equation} \label{equationB}
|\L^{\sharp,q-1}_{R_j}| \gtrsim C |F|^{\omega + (d-k) - (k-q)}
\end{equation}
assuming \eqref{firstiterationrequirement}. 

It follows from the definition of $\L^{\sharp,q-1}$ that for each $L \in \L^{\sharp,q-1}_{R_j}$
\[
|\{P \in \P^{\sharp,q-1}: P \subset L\}| \gtrsim |F|^k.
\]
Thus, using Lemma \ref{sccprop} to obtain
\[
|\{P \subset L : (L \cap R_j) \subset P\}| \lesssim |F|^{k-q} \ll |F|^k
\]
and (for $q > 1$) Lemma \ref{eiprop} to obtain
\begin{equation} \label{eipropapp}
|\{P \subset L : (L \cap R_j) \cap P = \emptyset\}| \lesssim |F|^{k-q+1} \ll |F|^k
\end{equation}
it follows that
\begin{equation} \label{rieqn}
|\{P \in \P^{\sharp,q-1}_{R_j}: P \subset L\}| \gtrsim |F|^k.
\end{equation}
Estimate \eqref{rieqn} also holds when $q=1$, since the special definition of $\P^{\sharp,0}_{R_j}$ means that we do not need to use \eqref{eipropapp}.

Applying Lemma \ref{folprop}
gives $(m + k -q  + 1)$-planes $T_i$ containing $R_j$. Let
\begin{align*}
\L_i &= \{L \in \L^{\sharp,q-1}_{R_j} : L \subset T_i\} \\ 
\P_i &= \bigcup_{L \in \L_i}\{P \in \P^{\sharp,q-1}_{R_j}: P \subset L\}.
\end{align*}
Then assuming 
\begin{equation} \label{seconditerationrequirement}
m + k -q + 1\leq d
\end{equation}
we have
\begin{align}
\nonumber |\P^{\sharp,q-1}_{R_j}| &\geq \sum_i |\P_i| \\ 
\nonumber&\gtrsim \sum_i |\L_i|/|F|^{m-q-k} \\ 
\nonumber&\geq |\L^{\sharp,q-1}_{R_j}|/|F|^{m-q - k}\\ 
\label{equationstar}&\gtrsim C |F|^{\omega + d - (m+k-q) + q}
\end{align}
where, for the second inequality, we used the fact (which follows from Lemma \ref{pslemma}) that $|\L_i| < |F|^{(k+1)(m -q) + k} $ to see that for some $d' \leq m+k-q$ we can estimate $|\P_i|$ using the previously known case $d'$ of the theorem.

Starting with \eqref{lequalszeroP} and iterating the fact that \eqref{equationA} implies \eqref{equationB} and \eqref{equationstar}, we obtain for $0 \leq q \leq k$
\[
|\P^{\sharp,q}_{R_j}| \gtrsim C^{k-q} |F|^{k + q(m-k) + (k-q)d - (k-q)(k-q-1)}
\]
if $m + k-q \leq d$ and
\[
|\L^{\sharp,q}_{R_j}| \gtrsim C^{k-q} |F|^{k + (q+1)(m-k) + (k-q-1)d - (k-q-1)(k-q-2) + (d - k) - (k-q-1)} 
\]
if $m + k-q - 1 \leq d.$

So if $m+k \leq d$ we have
\[
|\P| \geq |\P^{\sharp,0}_{R_j}| \gtrsim |F|^{k(d - k + 1) + k}
\]
as desired, and otherwise we have
\begin{equation} \label{toomanykplanes}
|\L| \geq |\L^{\sharp,k-(d + 1 - m)}_{R_j}| \gtrsim C^{d+1-m} |F|^{(k+1)(d-k) + k}.
\end{equation}
Choosing $C$ sufficiently large depending on the relevant implicit constants (none of which depended on $C$), we see that \eqref{toomanykplanes} contradicts the assumption \eqref{refinedk} and so we must have \eqref{firstpossibility}.
\end{proof}

\bibliographystyle{amsplain}
\bibliography{roberlin}
\end{document}